\documentclass{article}
\usepackage{upgreek}
\usepackage{latexsym,epsfig,bm,amssymb}
\usepackage{color}
\usepackage{mathptmx,amsthm,mathrsfs}
\input ot1ptm.fd \input ot1ztmcm.fd \input omlztmcm.fd
\input omsztmcm.fd  \input omxztmcm.fd \input ulasy.fd \input ursfs.fd

\def\version{}

\newcommand{\notyet}[1]{}

\DeclareSymbolFont{AMSb}{U}{msb}{m}{n}
\DeclareSymbolFontAlphabet{\mathbb}{AMSb}

\newcommand\pa{\partial}

\newcommand\ov{\overline}

\def\Re{{\rm Re\, }}
\def\Im{{\rm Im\,}}
\providecommand\C{{\mathbb C}}
\renewcommand\C{{\mathbb C}}
\newcommand{\R}{{\mathbb R}}

\newcommand{\abs}[1]{\vert #1 \vert}

\newcommand{\norm}[1]{\Vert #1 \Vert}
\newcommand{\const}{{\rm const}}

\providecommand{\ltor}[1]{
\ifnum #1=1{\it i}\else\ifnum #1=2{\it ii}\else\ifnum #1=3{\it iii}
\else\ifnum #1=4 {\it iv}\fi\fi\fi\fi
}

\DeclareMathSymbol{\varPhi}{\mathord}{letters}{"08}
\DeclareMathSymbol{\varOmega}{\mathord}{letters}{"0A}

\font\thf cmssdc10 at 11pt

\theoremstyle{plain}
\newtheorem{theorem}{\thf Theorem}[section]

\newtheorem{lemma}[theorem]{\thf Lemma}
\newtheorem{corollary}[theorem]{\thf Corollary}
\newtheorem{proposition}[theorem]{\thf Proposition}

\theoremstyle{definition}
\newtheorem{definition}[theorem]{Definition}
\newtheorem{remark}[theorem]{\thf Remark}

\makeatletter\@addtoreset{equation}{section}
\makeatother

\begin{document}

\title{On global well-posedness for 
Klein-Gordon equation with concentrated nonlinearity
}

\author{
{\sc Elena Kopylova}
\footnote{Research supported by the Austrian Science Fund (FWF) under Grant No. P27492-N25 
and RFBR grant No. 16-01-00100a}
\\ 
{\it\small Faculty of Mathematics of Vienna University and IITP RAS
}}

\date{\version}

\maketitle

\begin{abstract}
We prove global well-posedness for the 3D Klein-Gordon equation with a concentrated nonlinearity.
\end{abstract}

%%%%%%%%%%%%%%%%%%%%%%%%%%%%%%%%%%%%%%%%%%%%%%%%%%%%%%%%%%%%%%%%%%%%%%%%%%%%%%%%%%%%%%%%%%%%%%%%%%%%%%%%%
\section{Introduction}
\label{int-results}
%%%%%%%%%%%%%%%%%%%%%%%%%%%%%%%%%%%%%%%%%%%%%%%%%%%%%%%%%%%%%%%%%%%%%%%%%%%%%%%%%%%%%%%%%%%%%%%%%%%%%%%%%
%%%%%%%%%%%%%%%%%%%%%%%%%%%%%%%%%%%%%%%%%%%%%%%%%%%%%%%%%%%%%%%%%%%%%%%%%%%%%%%%%%%%%%%%%%%%%%%%%%%%%%%%%
The paper concerns  a nonlinear interaction of the Klein-Gordon field with 
point oscillators. The system is governed by  the following equations
\begin{equation}\label{iKG}
\left\{\begin{array}{c}
\ddot \psi(x,t)=(\Delta-m^2)\psi(x,t)+\sum\limits_{1\le j\le n}\zeta_j(t)\delta(x-y_j)\\\\
\lim\limits_{x\to y_j}(\psi(x,t)-\zeta_j(t)g_j(x))=F_j(\zeta(t)),\quad 1\le j\le n
\end{array}\right|\quad y_j\in\R^3,\quad x\in\R^3,\quad t\in\R,
\end{equation} 
with $m>0$ and $\zeta(t)=(\zeta_1(t),...,\zeta_n(t))\in\C^n$. 
Here  $g_{j}(x)=g(x-y_j)$, and $g(x)$ is the Green's function of the operator $-\Delta+m^2$ in $\R^3$, i.e.,
\begin{equation}\label{Green}
  g(x)=\frac{e^{-m|x|}}{4\pi|x|}.
\end{equation} 
The nonlinearity $F(\zeta)=(F_1(\zeta),...,F_n(\zeta))$ admits a real-valued potential:
\begin{equation}\label{FU}
 F(\zeta)=\pa_{\ov\zeta} U(\zeta),\quad U\in C^2(\C^n),
\end{equation} 
 where  $\pa_{\ov\zeta_j}:=\frac12(\frac{\pa U}{\pa\zeta_{j1}}+i\frac{\pa U}{\pa\zeta_{j2}})$ with $\zeta_{j1}:=\Re\zeta_j$ and $\zeta_{j2}:=\Im\zeta_j$.
Let $G=\{g_{jk}\}$ be a matrix with the entries 
\begin{equation}\label{Gjk}
g_{jk}:=\left\{\begin{array}{c}
  \frac{e^{-m|y_j-y_k|}}{4\pi|y_j-y_k|},~~{\rm if}~~j\not =k
  \\\\
  0,\qquad~~{\rm if}~~j =k
 \end{array}\right.
\end{equation}
and let ${\mathcal G}(\zeta)=(G\zeta,\zeta)=\sum\limits_{1\le k,j\le n}g_{jk}\zeta_j\overline\zeta_k$.    
We  assume that $U(\zeta)$ is such that
\begin{equation}\label{bound-below}
U(\zeta)-{\mathcal G}(\zeta)\ge b\abs{\zeta}^2-a, \quad{\rm for}\ \zeta\in\C^n,\quad
{\rm where}\ b> 0~~~{\rm and}~~~a\in\R.
\end{equation}
Our main result is the following. For the initial data of type
\begin{equation}\label{in_d}
\psi(x,0)=\psi_0(x)+\sum\limits_{1\le j\le n}\zeta_{0j}g_j(x),\quad
\dot\psi(x,0)=\dot\psi_0(x)+\sum\limits_{1\le j\le n}\dot\zeta_{0j}g_j(x),
\end{equation}
where
$\psi_0\in H^2(\R^3)$ and $\dot\psi_0\in H^1(\R^3)$, we prove
a global well-posedness of the Cauchy problem for  (\ref{iKG})
(see Theorem \ref{theorem-well-posedness} below).

In the context of the 3D Schr\"odinger and wave equations the point interaction of type (\ref{iKG})
was introduced in \cite{AAFT} -\cite{ AGHH}, \cite{KP}-\cite{NP1},
where the well-posedness, blow-up and asymptotic stability of solutions  was studied.
The first justification of the model with a nonlinear point interaction was given in the NLS case 
in the recent paper \cite{CFNT}.
The  nonlinear Schr\"odinger dynamics with a nonlinearity concentrated at a point is
obtained in \cite{CFNT}  as a scaling limit of a regularized nonlinear Schr\"odinger dynamics.

However, for the 3D Klein-Gordon equation with a point interaction the global well-posedness
and a justification of the model has not been obtained. In present paper we concentrate on  the first problem only.
We suppose that a justification can be done by suitable modification of methods \cite{CFNT}, but
it still remains an open question.

Let us comment on our approach. 
We develop for the  Klein-Gordon equation the approach suggested in \cite{NP} for wave equation.
First we consider  the linear system (\ref{iKG}) with $n=1$ and
$F_1(\zeta_1)=\zeta_1$. In this case  the solution $\psi$ can be represent as a sum:
\[
\psi(x,t)=\psi_f(x,t)+\varphi(x,t),
\]
where $\psi_f(x,t)$ is a solution to the Cauchy problem for the free Klein-Gordon equation 
with the  initial data (\ref{in_d}), and
$\varphi(x,t)$ is a solution to the coupled system  of i) the Klein-Gordon equation with delta-like source
and with zero initial data, and of ii)  the first-order linear integro-differential equation  
which control the dynamics of the coefficients $\zeta_1(t)$.
As a consequence, we derive an important regularity  property of  $\psi_f(y_1,t)$ 
(see Proposition \ref{KurasovKG}, cf. also \cite [Lemma 3.4]{NP}). 

We use this regularity property for proving the existence of a local solution to (\ref{iKG})
in the case of nonlinear function $F(\zeta)$. 
Then we obtain the energy conservation  for the dynamics (\ref{iKG}). 
Finally, we use the energy conservation to obtain a global existence theorem. 

The Klein-Gordon equation differs from the wave equation, considered in \cite{KP}--\cite{NP1},  
by the absence of strong Huygens principle. This result in additional integral  terms
(convolutions with  Bessel functions) in many calculations. 

We expect that the  result and methods of present paper  will be useful for the theory of attractors for
${\rm U}(1)$-invariant Hamiltonian system (\ref{iKG}) 
(cf. \cite{ANO}, \cite{BKKS}, \cite{KK07}, \cite{KK09}). 

Our paper is organized as follows. In Section \ref{results} we formulate the main theorem.
In Section \ref{lin-sect} we study the structure of  solutions to the  Klein-Gordon equation
with linear one point interaction. In Section \ref{sect-psifreg} we prove the key
regularity  property of  solutions to the free  Klein-Gordon equation with initial data (\ref{in_d}).
In section \ref{KGD-sect} we derive some preliminary formulas.
In Section \ref{nonlin-sect} we consider the nonlinear equation and  prove the main theorem.

%%%%%%%%%%%%%%%%%%%%%%%%%%%%%%%%%%%%%%%%%%%%%%%%%%%%%%%%%%%%%%%%%%%%%%%%%%%%%%%%%%%%%%%%%%%%%%%%%%%%%%%%%
\section{Main result}
\label{results}
%%%%%%%%%%%%%%%%%%%%%%%%%%%%%%%%%%%%%%%%%%%%%%%%%%%%%%%%%%%%%%%%%%%%%%%%%%%%%%%%%%%%%%%%%%%%%%%%%%%%%%%%%
We  fix a nonlinear function $F:\C^n\to\C^n$ and define the domain
\begin{equation}\label{q}
D_F=\{\psi\in L^2(\R^3):\psi(x)=\psi_{reg}(x)
+\sum\limits_{1\le j\le n}\zeta_j g_j(x),~~\psi_{reg}\in H^2(\R^3),~~
 \zeta_j\in\C,~~\lim\limits_{x\to y_j}(\psi(x)-\zeta_jg_j(x))=F_j(\zeta)
 \},
\end{equation}
which generally is not a linear space. Note that the last condition in (\ref{q})
is equivalent to
\begin{equation}\label{q1}
\psi_{reg}(y_j)+\sum\limits_{1\le k\le n}g_{kj}\zeta_k=F_j(\zeta).
\end{equation}
Let $H_F$ be a nonlinear operator on the domain $D_F$ defined by 
\begin{equation}\label{HF}
 H_F \psi=(\Delta-m^2)\psi_{reg},\quad\psi\in D_F.
\end{equation}
The system (\ref{iKG}) for  $\psi(t)\in D_F$ reads
\begin{equation}\label{KG}
\ddot \psi(x,t)=H_F \psi(x,t),\quad x\in\R^3,\quad t\in\R.
\end{equation}
Let us introduce a phase space for equation (\ref{HF}). Denote 
\[
\dot D=\{\pi\in L^2(\R^3):\pi(x)=\pi_{reg}(x)+\sum\limits_{1\le j\le n}\eta_j g_j(x),
~~\pi_{reg}\in H^1(\R^3), ~~\eta_j\in\C\}.
\]
Obviously, $D_F\subset\dot D$.
%%%%%%%%%%%%%%%%%%%%%%%%%%%%%%%%%%%%%%%%%%%%%%%%%%%%%%%%%%%%%%%%%%%%%%%%%%%%%%%%%%%%%%%%%%%%%%%%%
\begin{definition}
\begin{enumerate}
\item
${\cal D}_F$ is the Hilbert space of  states 
$\Psi=(\psi(x),\pi(x))\in D_F\oplus\dot D$ equipped with the finite norm
\begin{equation}\label{def-e}
\Vert\Psi\Vert_{{\cal D}_F}^2:=\norm{\psi_{reg}}_{H^2(\R^3)}^2+
\norm{\pi_{reg}}_{H^1(\R^3)}^2+\sum\limits_{1\le j\le n}|\zeta_j|^2+\sum\limits_{1\le j\le n}|\eta_j|^2.
\end{equation}
\item
${\cal X}$ is the Hilbert space of the states 
$\Psi=(\psi(x),\pi(x))\in H^2(\R^3)\oplus H^1(\R^3)$ equipped with the finite  norm
\begin{equation}\label{def-e1}
\Vert\Psi\Vert_{\cal X}^2:=\norm{\psi}_{H^2(\R^3)}^2+
\norm{\pi}_{H^1(\R^3)}^2.
\end{equation}
\end{enumerate}
\end{definition}
%%%%%%%%%%%%%%%%%%%%%
Denote $\Vert\cdot\Vert=\Vert\cdot\Vert_{L^2(\R^3)}$.
Our main result is the following.
%%%%%%%%%%%%%%%%%%%%%%%%%%%%%%%%%%%%%%%%%%%%%%%%%%%%%%%%%%%%%%%%%%%%%%%%%%%%%%%%%
\begin{theorem}\label{theorem-well-posedness}
Let  conditions (\ref{FU}) and (\ref{bound-below}) hold. Then 
\begin{enumerate}
\item
For every initial data $\Psi_0=(\psi\sb 0,\pi\sb 0)\in {\cal D}_F$  the equation (\ref{HF})
has a unique solution $\psi(t)$ such that 
\[
\Psi(t)=(\psi(t),\dot\psi(t))\in C(\R,{\cal D}_F).
\]
\item
The map $W(t):\;\Psi_0\mapsto \Psi(t)$ is continuous in ${\cal D}_F$ for each $t\in\R$.
\item
The energy is conserved, i.e.,
\begin{equation}\label{ec}
{\cal H}_F(\Psi(t)):=
\Vert\dot\psi(t)\Vert^2+\Vert\nabla\psi_{reg}(t)\Vert^2
+m^2\Vert\psi_{reg}(t)\Vert^2+U(\zeta(t))-{\cal G}(\zeta(t))=\const, \quad t\in\R.
\end{equation}
\item
The following a priori bound holds
\begin{equation}\label{apb}
|\zeta(t)|\le C(\Psi_0),\quad t\in\R. 
\end{equation}
\end{enumerate}
\end{theorem}
Obviously, it suffices to prove Theorem \ref {theorem-well-posedness} for $t\ge 0$.

%%%%%%%%%%%%%%%%%%%%%%%%%%%%%%%%%%%%%%%%%%%%%%%%%%%%%%%%%%%%%%%%%%%%%%%%%%%%%%%%%%%%%%%%%%%%%%%5
\section{Linear equation:  structure of solution}
\label{lin-sect}
%%%%%%%%%%%%%%%%%%%%%%%%%%%%%%%%%%%%%%%%%%%%%%%%%%%%%%%%%%%%%%%%%%%%%%%%%%%%%%%%%%%%%%%%%%%%%%%%
%%%%%%%%%%%%%%%%%%%%%%%%%%%%%%%%%%%%%%%%%%%%%%%%%%%%%%%%%%%%%%%%%%%%%%%%%%%%%%%%%%%%%%%%%%%%%%%%
Here we consider the  Klein-Gordon equation with a linear one point interaction, i.e. $n=1$.
More precisely, denote 
\[
D_y=\{\psi\in L^2(\R^3):\psi=\psi_{reg}+\xi g(x-y),~~\psi_{reg}\in H^2(\R^3),~~
\psi_{reg}(y)=\xi\in\C\},
\]
\[
\dot D_y=\{\pi\in L^2(\R^3):\pi(x)=\pi_{reg}(x)+\chi g(x-y),
~~\pi_{reg}\in H^1(\R^3), ~~\chi\in\C\},
\]
and consider the operator
\begin{equation}\label{Hl}
H_y\psi:=(\Delta-m^2)\psi_{reg},\quad\psi\in D_y.
\end{equation}
%%%%%%%%%%%
\begin{proposition}\label{KurasovLin}
Let $\psi_0 =\psi_{0,reg}+\xi_{0}g(x-y)\in D_y$ 
and $\pi_0=\pi_{0,reg}+\dot\xi_{0}g(x-y) \in\dot D_y$.
Then the Cauchy problem
\begin{equation}\label{CPL}
\ddot{\psi}(x,t) = H_y\psi(x,t),\quad \psi(x,0) = \psi_0(x),\quad\dot\psi(x,0)  =  \pi_0(x),
\end{equation}
has a unique strong solution $\psi(t) \in C(\R,D_y)\cap C^1(\R,\dot D_y)\cap C^2(\R,L^2(\R^3))$.
\end{proposition}
%%%%%%%%%%%%%%%%%%%%%%%%%%%%%%%%%%%%%%%%%%%%%%%%%%%%%%%%%%%%%%%%%%%%%%%%%%%%%%%%%%%%%%%%%%%%%%%%%%%
\begin{proof}
The operator $H_y$ is a symmetric  operator, since 
$$
\langle H_y\psi,\varphi\rangle=
\langle (\Delta-m^2)\psi_{reg},\varphi_{reg}
+\xi_{\varphi}g(\cdot-y)\rangle
=\langle \psi_{reg},(\Delta-m^2)\varphi_{reg}\rangle
-\xi_\psi\overline\xi_{\varphi}=\langle \psi,H_y\varphi\rangle.
$$
Moreover,
$$
\langle -H_y\psi,\psi\rangle=\Vert\nabla\psi_{reg}\Vert^2+m^2\Vert\psi_{reg}\Vert^2
+\xi_\psi\overline\xi_{\psi}\ge 0.
$$
Hence, $H_y$  admits a unique selfajoint extension, and 
the  corresponding Cauchy problem has a unique strong solution in
$C(\R,D_y)\cap C^1(\R,\dot D_y)\cap C^2(\R,L^2(\R^3))$
by the theory of abstract wave equations in Hilbert spaces 
(see e.g. \cite[Chapter 2, Section 7]{G}). 
\end{proof}
%%%%%%%%%%%%%%%%%%%%%%%%%%%%%%%%%%%%%%%%%%%%%%%%%%%%%%%%%%%%%%%%%%%%%%%%%%%%%%%%%%%%%%%%%%%%%
Proposition \ref{KurasovLin} implies that the strong solution $\psi(x,t)$ 
of (\ref{CPL}) satisfies
\begin{equation}\label{sol2}
\psi(x,t)=\psi_{reg}(x,t)+\xi(t)g(x-y),\quad\psi_{reg}(y,t)=\xi(t)
\end{equation}
with $\psi_{reg}(t)\in C(\R,H^2(\R^3))$ and $\xi\in C^1(\R)$.
Now we obtain an integral representation for  
$\psi(x,t)$ via a solution to integro-differential equation.
%%%%%%%%%%%%%%%%%%%%%%%%%%%%%%%%%%%%%%%%%%%%%%%%%%%%%%%%%%%%%%%%%%%%%%%%%%%%%%%%%%%%%%%%%%%%%%%%%%%%%%%%%%%%
\begin{lemma}\label{KurasovLin1} (cf. \cite [Theorem 3]{KP}).
The strong solution $\psi(x,t) \in C(\R,D_y)\cap C^1(\R,\dot D_y)\cap C^2(\R,L^2(\R^3))$ of the Cauchy problem (\ref{CPL}) for $t\ge 0$
is  given by
\begin{equation}\label{sol11}
\psi(x,t)=\psi_f(x,t)+\frac{\theta(t-|x-y|)}{4\pi|x-y|}\xi(t-|x-y|)-\frac{m}{4\pi}
\int_0^t\frac{\theta(s-|x-y|)J_1(m\sqrt{s^2-|x-y|^2})}{\sqrt{s^2-|x-y|^2}}\xi(t-s)ds.
\end{equation}
Here $\psi_f(x,t)\in C([0,\infty),L^2(\R^3))$ is the unique solution 
to the Cauchy problem for the free Klein-Gordon equation
\begin{equation}\label{CP1}
\ddot{\psi}_f(x,t) = (\Delta-m^2)\psi_f(x,t),
\quad \psi_f(x,0) = \psi_0(x),\quad\dot\psi_f(x,0)  =  \pi_0(x),
\end{equation}
and  $\xi(t)\in C^1([0,\infty))$ is the unique solution to the Cauchy problem 
for the following first-order linear  integro-differential equation with delay
\begin{equation}\label{zeta-sol}
\frac {1}{4\pi}(\dot\xi(t)-m\xi(t))+\xi(t)-\lambda(t)+\frac{m}{4\pi}
\int_0^t\frac{J_1(m(t-s))}{t-s}\xi(s)ds=0,\quad \xi(0)=\xi_{0},\quad t\ge 0,
\end{equation}
where $\lambda(t)\in C([0,\infty))$, and $\lambda(t):=\lim\limits_{x\to y}\psi_f(x,t)$ for $t>0$.
\end{lemma}
%%%%%%%%%%%%%%%%%%%%%%%%%%%%%%%%%%%%%%%%%%%%%%%%%%%%%%%%%%%%%%%%%%%%%%%%%%%%%%%%%%%%%%%%%%%%%%%%%%%%
\begin{proof}
In notation (\ref{sol11}) define the function
\begin{equation}\label{pp-rep}
\varphi(x,t):=\frac{\theta(t-|x-y|)}{4\pi|x-y|}\xi(t-|x-y|)-\frac{m}{4\pi}
\int_0^t\frac{\theta(s-|x-y|)J_1(m\sqrt{s^2-|x-y|^2})}
{\sqrt{s^2-|x-y|^2}}\xi(t-s)ds.
\end{equation}
It is easy to verify that 
$\varphi(t)\in C([0,\infty),L^2(\R^3))$ and satisfy the equation
\begin{equation}\label{CP2}
\ddot{\varphi}(x,t)= (\Delta-m^2) \varphi(x,t) +\xi(t)\delta(x-y),
\quad \varphi(x,0) = 0,\quad\dot\varphi(x,0)=0.
\end{equation}
Hence, for $\psi_f=\psi-\varphi\in C([0,\infty),L^2(\R^3))$,  we obtain
\begin{eqnarray*}
&&\ddot{\psi}_f(x,t)=\ddot\psi(x,t)-\ddot\varphi(x,t)
=(\Delta-m^2)\psi_{reg}(x,t)-\Big((\Delta-m^2)\varphi(x,t)+\xi(t)\delta(x-y)\Big)\\
\nonumber
&&=(\Delta-m^2)\Big(\psi(x,t)-\xi(t)g(x-y)\Big)
-\Big((\Delta-m^2)\varphi(x,t+\xi(t)\delta(x)\Big)\\
\nonumber
&&=(\Delta-m^2)\Big(\psi(x,t)-\varphi(x,t)\Big)=(\Delta-m^2)\psi_f(x,t),
\end{eqnarray*}
and thus $\psi_f$ is the solution to the Cauchy problem (\ref{CP1}).
\smallskip\\
Let us prove the existence and continuity of $\lambda(t)=\lim\limits_{x\to y}\psi_{f}(x,t)$.
We can split  $\psi_f(x,t)$ as
\[
\psi_f(x,t)=\psi_{f,reg}(x,t)+\psi_{f,y}(x,t),
\]
where $\psi_{f,reg}$ is the solution to the free Klein-Gordon equation with regular initial data
$\psi_{0,reg}$, $\pi_{0,reg}$, and $\psi_{f,y}$ are the solutions to the free Klein-Gordon equation 
with initial data $\xi_{0} g(x-y)$, $\dot\xi_{0} g(x-y)$.
%%%%%%%%%%%%%%%%%%%%%%%%%%%%%%%%%%%%%%%%%%%%%%%%%
Evidently, $\psi_{f,reg}(x,t)\in C([0,\infty), H^2(\R^3))$ and there exists
\begin{equation}\label{lamreg}
\lambda_{reg}(t):=\lim\limits_{x\to y}\psi_{f,reg}(x,t)=\psi_{f,reg}(y,t)\in C([0,\infty)).
\end{equation}
Now consider $\psi_{f,y}(x,t)$. Note that the function
\[
\eta_y(x,t):=\psi_{f,y}(x,t)-(\xi_{0}+t\dot\xi_{0})\,g(x-y)
\] 
satisfies 
\[
\ddot\eta_y(x,t)=(\Delta-m^2)\eta_y(x,t)-(\xi_{0}+t\dot\xi_{0})\,\delta(x-y)
\]
with zero initial data. Therefore, one obtains
\begin{eqnarray*}
\eta_y(x,t)
&=&-\int_0^t\Big(\frac{\delta(t-s-|x-y|)}{4\pi (t-s)}
-\frac{m}{4\pi}\frac{\theta(t-s-|x-y|)J_1(m\sqrt{(t-s)^2-|x-y|^2})}
{\sqrt{(t-s)^2-|x-y|^2}}\Big)(\xi_{0}+s\dot\xi_{0})ds\\
&=&-\frac{\theta(t-|x-y|)(\xi_{0k}+(t-|x-y|)\dot\xi_{0})}{4\pi|x-y|}\\
&+&\frac{m}{4\pi}\int_0^t\frac{\theta(t-s-|x-y|)
J_1(m\sqrt{(t-s)^2-|x-y|^2})}{\sqrt{(t-s)^2-|x-y|^2}}(\xi_{0}+s\dot\xi_{0})ds.
\end{eqnarray*}
Hence, 
\begin{eqnarray}\nonumber
\psi_{f,y}(x,t)&=&-\frac{\theta(t-|x-y|)(\xi_{0}+(t-|x-y|)\dot\xi_{0})}{4\pi|x-y|}
+\frac{(\xi_{0}+t\dot\xi_{0})e^{-m|x-y|}}{4\pi|x-y|}\\
\label{psify}
&+&\frac{m}{4\pi}\int_0^t\frac{\theta(t-s-|x-y|)
J_1(m\sqrt{(t-s)^2-|x-y|^2})}{\sqrt{(t-s)^2-|x-y|^2}}(\xi_{0}+s\dot\xi_{0})ds,\quad t\ge 0.
\end{eqnarray}
Therefore, for any $t>0$ there exists 
\begin{equation}\label{lamG}
\lambda_{y}(t)=\lim_{x \to y}\,\psi_{f,y}(x,t)=
-\frac{m(\xi_{0}+t\dot\xi_{0})-\dot\xi_0}{4\pi}
+\frac{m}{4\pi}\int_0^t \frac{J_1(m(t-s))}{(t-s)}(\xi_{0}+s\dot\xi_{0})ds,
\end{equation}
and we set
\begin{equation}\label{lamG0}
\lambda_{y}(0)=\lim_{t \to 0}\lambda_{y}(t)=-\frac{m\xi_{0}-\dot\xi_{0}}{4\pi}.
\end{equation}
Thus, (\ref{lamreg})--(\ref{lamG0}) imply that
\[
\lambda(t)=\lambda_{reg}(t)+\lambda_{y}(t)\in C([0,\infty)).
\]
It remains to prove that $\xi(t)$ is  a solution to the Cauchy problem (\ref{zeta-sol}). 
Writing the second equation of (\ref{sol2})  in terms of the decomposition 
(\ref{sol11}), we obtain
\begin{eqnarray}\nonumber
\xi(t)\!\!\!&&\!\!\!
=\lim_{x \to y}\, ( \psi(t,x)-\xi(t) g(x-y))
=\lim_{x \to y}\, \big( \psi_f(t,x)+\varphi(x,t)-\xi(t) g(x-y)\big)\\
\nonumber
\!\!\!&&\!\!\!=\lambda(t)+
\lim_{x\to y}\big(\frac{\theta(t-|x-y|)\xi(t-|x-y|)}{4\pi|x-y|}-\frac{\xi(t)e^{-m|x-y|}}{4\pi|x-y|}\big)\\
\nonumber
\!\!\!&&\!\!\!-\lim_{x\to y}\frac{m}{4\pi}
\int_0^t\frac{\theta(t-s-|x-y|)J_1(m\sqrt{(t-s)^2-|x-y|^2})}{\sqrt{(t-s)^2-|x-y|^2}}\xi(s)ds\Big)
\\
\label{lim_z}
\!\!\!&&\!\!\!=\lambda(t)-\frac{1}{4\pi}(\dot\xi(t)-m\xi(t))
-\frac{m}{4\pi}\int_0^t\frac{J_1(m(t-s)}{t-s}\xi(s)ds,
\quad t\ge0.
\end{eqnarray}
Hence $\xi(t)$ satisfies (\ref{zeta-sol}). 
\end{proof}

%%%%%%%%%%%%%%%%%%%%%%%%%%%%%%%%%%%%%%%%%%%%%%%%%%%%%%%%%%%%%%%%%%%%%%%%%%%%%%%%%%%%%%%%%%%%%%%5
\section{Regularity property}
\label{sect-psifreg}
%%%%%%%%%%%%%%%%%%%%%%%%%%%%%%%%%%%%%%%%%%%%%%%%%%%%%%%%%%%%%%%%%%%%%%%%%%%%%%%%%%%%%%%%%%%%%%%%
Here we establish the key regularity property of solution $\psi_{f}(x,t)$ to the free Klein-Gordon equation 
with initial data from $\cal X$.
%%%%%%%%%%%%%%%%%%%%%%%%%%%%%%%%%%%%%%%%%%%%%%%%%%%%%%%%%%%%%%%%%%%%%%%%%%%%%%%%%%%%%%%%%%%%%%%%%
First  we prove an auxiliary lemma.
\begin{lemma}\label{e-1} (cf. \cite [Lemma 3.2]{NP}).
Let $0<T<\infty$, $y\in\R^3$, $f(t)\in C([0,T])$, and $h(t)=\dot f(t)$ in the sense of distribution.
Let $u(x,t)\in C([0,T],L^2(\R^3))$ be the solution to the  Cauchy problem
\begin{equation}\label{KGG}
\ddot{u}(x,t)  =  (\Delta-m^2) u(x,t) +h(t)\, g(x-y),
\quad u(x,0)=u_0(x),\quad \dot u(x,0)=\dot u_0(x)
\end{equation}
 with initial data  $(u_0, \dot u_0)\in {\cal X}$. Then
$h\in L^2([0,T])$ if and only if  $(u(t),\dot u(t)\in C([0,T],{\cal X})$.
\end{lemma}
%%%%%%%%%%%%%%%%%%%%%%%%%%%%%%%%%%%%%%%%%%%%%%%%%%%%%%%%%%%%%%%%%%%%%%%%%%%%%%%%%%%%%%%%
\begin{proof} 
It  suffices to consider  $y=0$.
\smallskip\\
{\it i)}. Suppose that $h\in L^2([0,T])$.
We  represent $u(x,t)$ as the sum $u(x,t)=u_1(x,t)+u_2(x,t)$,
where $u_1(x,t)$ is a solution to the free Klein-Gordon equation 
with the initial data $u_0,\dot u_0$, and
$u_2(x,t)$ is a solution to (\ref{KGG}) with zero initial data. 
Evidently, $(u_1(t),\dot u_1(t))\in C([0,\infty),{\cal X})$.
Let us prove that
\begin{equation}\label{u2}
(u_2(t),\dot u_2(t))\in C([0,T],{\cal X}).
\end{equation}
Applying the Fourier  transform, we obtain
\[
(|k|^2+m^2)\,\tilde{u_2}(k,t)=
\int_0^t\frac{\sin((t-s)\sqrt{|k|^2+m^2})}{\sqrt{|k|^2+m^2}} h(s)ds,\quad
\sqrt{|k|^2+m^2}\,\tilde{\dot u_2}(k,t)=
\int_0^t\frac{\cos((t-s)\sqrt{|k|^2+m^2})}{\sqrt{|k|^2+m^2}} h(s)ds.
\]
Hence, for (\ref{u2}) it suffices to verify that for any $t\in [0,T]$,
\begin{eqnarray}\nonumber
\big\Vert\int_0^t\frac{\sin((t-s)\sqrt{|k|^2+m^2})}{\sqrt{|k|^2+m^2}} h(s)ds\big\Vert
&\le& C_1(t)\Vert h\Vert_{L^2([0,t])},\\
\label{uu2}
\big\Vert\int_0^t\frac{\cos((t-s)\sqrt{|k|^2+m^2})}{\sqrt{|k|^2+m^2}} h(s)ds\big\Vert
&\le& C_2(t)\Vert h\Vert_{L^2([0,t])}.
\end{eqnarray}
The both integrals are estimated in the same way, and we consider the first integral only.
We have 
\begin{eqnarray}\nonumber
&&\!\!\!\!\!\!\!\!\!\!\!\!\!\!\!\!\!\!\!\!\!\!\!\!\!\!\!\!\!\!\!\!\!
\big\Vert\int_0^t\frac{\sin((t-s)\sqrt{|k|^2+m^2})}{\sqrt{|k|^2+m^2}}h(s)ds\big\Vert^2\\
\nonumber
&=&4\pi\,\lim_{R\uparrow\infty}
\int_0^t\int_0^tds\,ds'\,\bar {h}(s) h(s')\int_0^R
dr\,\frac{r^2\sin((t-s)\sqrt{r^2+m^2})\,\sin((t-s')\sqrt{r^2+m^2})}{r^2+m^2}\\
\nonumber
&=&4\pi\,\lim_{R\uparrow\infty}
\int_0^t\int_0^tds\,ds'\,\bar {h}(s) h(s')\int_0^R
dr\,\sin((t-s)\sqrt{r^2+m^2})\,\sin((t-s')\sqrt{r^2+m^2})\\
\label{dot}
&-&4\pi m^2\,\int_0^t\int_0^tds\,ds'\,\bar {h}(s) h(s')\int_0^\infty
dr\,\frac{\sin((t-s)\sqrt{r^2+m^2})\,\sin((t-s')\sqrt{r^2+m^2})}{r^2+m^2}
=I_1(t)+I_2(t)
\end{eqnarray}
It is easy to prove that
\begin{equation}\label{I2}
|I_2(t)|\le Ct\Vert h\Vert_{L^2([0,t])}^2,\quad t\in [0,T].
\end{equation}
It remains to estimate the term $I_1$. We have 
\begin{eqnarray}\label{I1}
I_1(t)&=&\pi\,\lim_{R\uparrow\infty}
\int_0^t\int_0^tds\,ds'\,\bar {h}(s) h(s')\int_{-R}^Rdr
\Big(e^{i(s-s')\sqrt{r^2+m^2}}+e^{i(2t-s-s')\sqrt{r^2+m^2}}\Big)\\
\nonumber
&=&\pi\sqrt{2\pi}\int_0^t ds\,|h(s)|^2+
\pi\int_0^t\int_0^tds\,ds'\,\bar{h}(s)h(s')\big(F(s-s')+F(2t-s-s')\big),
\end{eqnarray}
where
\[
F(z)=\int_{-\infty}^\infty dr (e^{iz\sqrt{r^2+m^2}}-e^{izr}).
\]
Note that
\[
e^{iz\sqrt{r^2+m^2}}-e^{izr}=e^{izr}\Big(e^{izr(\sqrt{1+(\frac{m}{r})^2}-1)}-1\Big)
=e^{izr}\Big(\frac  {izm^2}{2r}+R(r,z)\Big),\quad |r|\ge 2m^2+1,
\]
where $|R(r,z)|\le \frac 14(1+|z|)^2m^4/r^2$.
Hence,
\begin{equation}\label{F-est}
|F(z)|\le |\int\limits_{|r|\le 2m^2+1}...|+|\int\limits_{|r|\ge 2m^2+1}...| 
\le C(m)(1+|z|)^2+|\int\limits_{|r|\ge 2m^2+1} \frac  {m^2}{2r}de^{izr}|
\le C_1(m)(1+|z|)^2.
\end{equation}
From (\ref{I1}) and (\ref{F-est}) it follows that
\begin{equation}\label{I11}
|I_1(t)|\le C(m)(1+t^3)\Vert h\Vert_{L^2([0,t])}^2,\quad t\in [0,T].
\end{equation}
Finally,   (\ref{dot}), (\ref{I2}) and (\ref{I11}) imply  (\ref{uu2}), and then  (\ref{u2}).
%%%%%%%%%%%%%%%%%%%%%%%%%%%%%%%%%%%%%%%%%%%%%%%%%%%%%%%%%%%%%%%%%%%%%%%%%%%
\smallskip\\
{\it  ii)}
Suppose now that  $(u(t),\dot u(t))\in C([0,T],{\cal X})$ and prove that $h\in L^2([0,T])$. 
As the first step, we will estimate $\Vert h\Vert_{L^2([0,T])}$ via 
\[
M_T:=\max\limits_{t\in [0,T]}\Vert (u(t),\dot u(t))\Vert_{\cal X},
\]
assuming that $\xi\in C^1([0,T])$.
For any fixed $\tau\in [0,T)$ and  $t\in [\tau, T]$, we
split $u(x,t)$ as $u(x.t)=u_{1,\tau}(x.t)+u_{2,\tau}(x.t)$,
where $u_{1,\tau}(x,t)$ is the solution to the free Klein-Gordon equation 
with  initial data $u(x,\tau),\dot u(x,\tau)$, and
$u_{2,\tau}(x,t)$ is the solution to (\ref{KGG}) with zero initial data at $t=\tau$.
By the energy conservation for the free   Klein-Gordon equation 
\[
\Vert (u_{1,\tau}(t),\dot u_{1,\tau}(t))\Vert_{\cal X}
\le C \Vert (u(\tau),\dot u(\tau))\Vert_{\cal X}\le CM_T,\quad t\in [\tau, T].
\]
Therefore,
\begin{equation}\label{u2-est}
\Vert (u_{2,\tau}(t),\dot u_{2,\tau}(t))\Vert_{\cal X}\le CM_T,\quad t\in [\tau, T].
\end{equation}
Since $g(x)=(-\Delta+m^2)^{-1}\delta(x)$, we have $(-\Delta+m^2)u_{2,\tau}(x,t)=v_\tau(x,t)$, 
where $v_\tau(t)\in C([\tau,T],L^2(\R^3))$ is the unique solution to
$$
\ddot v(x,t)= (\Delta-m^2)v(x,t) +h(t)\delta(x)
$$ 
with zero initial data at $t=\tau$. 
Estimate (\ref{u2-est}) implies
\begin{equation}\label{v-est}
\Vert v_\tau(t)\Vert\le CM_T,\quad t\in [\tau, T].
\end{equation}
Similarly  to (\ref{pp-rep})--(\ref{CP2}), we obtain
\begin{equation}\label{pp-rep2}
v_\tau(x,t)=\frac{\theta(t-\tau-|x|)}{4\pi|x|}h(t-|x|)
-\frac{m}{4\pi}p_\tau(|x|,t),\quad t\ge  \tau,
\end{equation}
where
\begin{equation}\label{pa}
p_\tau(r,t)=\int_\tau^t\frac{\theta(t-s-r)
J_1(m\sqrt{(t-s)^2-r^2})}{\sqrt{(t-s)^2-r^2}}h(s)ds,\quad r\ge 0.
\end{equation}
We have
\[
\Vert v_\tau(t)\Vert^2=\frac 1{4\pi}\int_0^{t-\tau}|h(t-r)|^2dr
+\frac{m^2}{4\pi}\int_0^{t-\tau} r^2|p_\tau(r,t)|^2dr
+\frac{m}{2\pi}\int_0^{t-\tau}h(t-r)p_\tau(r,t)r\,dr,\quad t\in [\tau,T].
\]
Therefore,
\begin{equation}\label{pp-int}
\Vert h\Vert_{L^2([\tau,t])}^2+2m\int_0^{t-\tau}h(t-r)p_\tau(r,t)r\,dr
\le 4\pi\Vert v_\tau(t)\Vert^2\le C_1M_T,\quad t\in [\tau,T].
\end{equation}
due to (\ref{v-est}), where $C_1$ does not depend on $\tau$ and $t$.
%%%%%%%%%%%%%%%%%%%%%%%%%%%%%%%%%%%%%%%%%%%%%%%%%%%%%%
Applying the Cauchy-Schwarz inequality to (\ref{pa}),  we obtain
\[
|p_\tau(r,t)|\le C\Vert h\Vert_{L^2([\tau,t])}
\Big(\int_r^\infty\frac{J_1^2(m\sqrt{s^2-r^2})}{\sqrt{s^2-r^2}}ds\Big)^{1/2}.
\]
The properties of Bessel function $J_1$ (see \cite{O}) imply the uniform estimates
\begin{eqnarray*}
\int_r^\infty\frac{J_1^2(m\sqrt{s^2-r^2})}{\sqrt{s^2-r^2}}ds\le C,\quad 0\le r.
\end{eqnarray*}
Hence,
\[
|p_\tau(r,t)|
\le C\Vert h\Vert_{L^2([\tau,t])},\quad 0\le r,\quad t\in [\tau,T],
\]
and the Cauchy-Schwarz inequality imply 
\[
2m\int_0^{t-\tau}|h(t-r)p_\tau(r,t)|rdr
\le C_0(t-\tau)^{3/2}\Vert h\Vert_{L^2([\tau,t])}^2,
\]
where $C_0$ does not depend on $\tau$ and $t$.
Substituting the last inequality into (\ref{pp-int}), we obtain
\[
(1-C_0(t-\tau)^{3/2})\Vert h\Vert_{L^2([\tau,t])}^2\le C_1M_T,\quad t\in [\tau,T].
\]
Therefore, $h\in L^2([\tau,\tau+t_0])$, where $t_0=\min\{T-\tau,(2C_0)^{-2/3}\}$.
Since $0\le \tau<T$ is arbitrary, we can split 
the interval $[0,T]$ as 
\[
[0,T]=[0,t_0]\cup[t_0,2t_0]\cup...\cup[nt_0,T]
\]
and obtain that
\begin{equation}\label{dotxi-est}
\Vert h\Vert_{L^2([0,T])}\le C\max\limits_{t\in [0,T]}\Vert (u(t),\dot u(t))\Vert_{\cal X}.
\end{equation}
%%%%%%%%%%%%%%%%%%%%%%%%%%%%%%%%%%%%%%%%%%%%%%%%%%%%%%%%%%%%%
Finally, we should consider general case $\xi\in C([0,T])$.
In this case we define smooths approximations  $\xi_{\epsilon}(t)$:
\[
\xi_{\epsilon} (t)=\xi*\rho_{\epsilon} (t)=\int_0^T \rho_{\epsilon}(s)\xi(t-s)ds\in C^\infty([0,T]),
\]
where $\rho_{\epsilon}(s)=\frac{1}{\epsilon}\rho(\frac{s}{\epsilon})$, and $\rho(s)$ is a smooth function
with support in $[-1,0]$ such that $\int\rho(s)ds=1$.
Let $u_{\epsilon}(x,t)$ be a solution to  (\ref{KGG}) with $h_{\epsilon}(t)$
instead of $h(t)$. Then  we have
\[
\Vert h_{\epsilon}\Vert_{L^2([0,T])}
\le C\max\limits_{t\in [0,T]}\Vert (u_{\epsilon}(t),\dot u_{\epsilon}(t))\Vert_{\cal X}.
\]
Taking the limit as $\epsilon\to 0$, we obtain (\ref{dotxi-est}) in general case.
\end{proof}
\vspace{+3mm}
%%%%%%%%
\begin{proposition}\label{KurasovKG}
Let $\psi_{f}(x,t)\in C([0,\infty),L^2(\R^3))$ be the unique  solution to  (\ref{CP1})
with initial data 
$\psi_0 \in D_y$ and $\pi_0\in\dot D_y$,
and let $\lambda(t)=\lim\limits_{x\to y}\psi_{f}(x,t)$.
Then $\dot\lambda\in L^2_{loc}([0,\infty))$. 
\end{proposition}
%%%%%%%%%%%%%%%%%%%%%%%%%%%%%%%%%%%%%%%%%%%%%%%%%%%%%%%%%%%%%%%%%%%%%%%%%%%%%%%%%%%%%%%%%%%%%%%%%%%%%%%%%%%%%%%%%%%%
\begin{proof}
Substituting  (\ref{sol2}) into  (\ref{CPL}), we obtain
\begin{eqnarray}\nonumber
&&\ddot\psi_{reg}(x,t)=(\Delta-m^2)\psi_{reg}(x,t)-\ddot\xi(t) g(x-y),\\
\nonumber
&&\psi_{reg}(x,0)=\psi_{0,reg},\quad \dot\psi_{reg}(x,0)=\pi_{0,reg}.
\end{eqnarray}
Note that $\xi\in C^1(\R)$ since 
$\|\dot\psi(t)|^2_{\dot D_y}:=\|\dot\psi_{reg}(t)\|_{H^1(\R^3)}^2+|\dot\xi(t)|^2$.
Lemma \ref{e-1} implies that $\ddot\xi\in L^2_{loc}([0,\infty))$.
Then $\dot\lambda\in L^2_{loc}([0,\infty))$ by (\ref{zeta-sol}).
\end{proof}
%%%%%%%%%%%%%%%%%%%%%%%%%%%%%%%%%%%%%%%%%%%%%%%%%
\begin{corollary}\label{imp}(cf. \cite[Lemma 3.4]{NP}).
Let $u(x,t)$ be the solution to the free Klein-Gordon equation with regular initial data
$(u_0,\dot u_0)\in {\cal X}$. Then for all $y\in\R^3$
\begin{equation}\label{new_pr}
\dot u(y,t)\in L^2_{loc}([0,\infty)).
\end{equation}
\end{corollary}
\begin{proof}
Consider the Cauchy problem (\ref{CPL}) with initial data 
$\psi_0(x)=u_0(x)+\xi_0g(x-y)$, $\pi_0(x)=\dot u_0(x)+\dot\xi_0g(x-y)$,
where $\xi_0=u_0(y)$ and $\dot\xi_0$ is arbitrary.
Then in notation of Lemma \ref{KurasovLin1} we have $u(x,t)=\psi_{f,reg}(x,t)$ and
$u(y,t)=\lambda_{reg}(t)=\lambda(t)-\lambda_y(t)$, where $\lambda_y(t)$ is defined by (\ref{lamG}).
Evidently, $\dot\lambda_y(t)\in L^2_{loc}([0,\infty))$. Then Proposition \ref{KurasovKG} implies (\ref{new_pr}).
\end{proof}
%%%%%%%%%%%%%%%%%%%%%%%%%%%%%%%%%%%%%%%%%%%%%%%%%%%%%%%%%%%%%%%%%%%%%%%%%%%%%%%%%%%%%%%%%%%%%%%%
\section{Free Klein-Gordon equation}
\label{KGD-sect}
Here we introduce some notations and obtain some formulas which we will use below.
We consider the free Klein-Gordon equation
\[
\ddot \psi_f(x,t)=(\Delta-m^2)\psi_f(x,t)
\]
with initial data
\[
\psi_0=\psi_{0,reg}+\sum\limits_{1\le k\le n}\zeta_{0k} g_k,\quad
\pi_0=\pi_{0,reg}+\sum\limits_{1\le k\le n}\dot\zeta_{0k} g_k.
\]
We  split  $\psi_f(x,t)$ as
\[
\psi_f(x,t)=\psi_{f,reg}(x,t)+\sum\limits_{1\le k\le n}\psi_{f,k}(x,t),
\]
where  $\psi_{f,reg}$ is the solution to the free Klein-Gordon equation with regular initial data
$\psi_{0,reg}$, $\pi_{0,reg}$, and $\psi_{f,k}$ are the solutions to the free Klein-Gordon equation 
with initial data $\zeta_{0k} g_k$, $\dot\zeta_{0k} g_k$.
Denote
\begin{equation}\label{lamregj}
\lambda_{j,reg}:=\psi_{f,reg}(y_j,t)\in C([0,\infty)),
\qquad\lambda_{j,k}(t)=\lim_{x \to y_j}\,\psi_{f,k}(x,t)
\end{equation}
Due to Corollary \ref{imp} 
\begin{equation}\label{dotlamreg}
\dot\lambda_{j,reg}(t)\in L^2_{loc}([0,\infty)).
\end{equation}
Further, we get  similarly to (\ref{psify})
\begin{eqnarray*}
\psi_{f,k}(x,t)&=&-\frac{\theta(t-|x-y_k|)(\zeta_{0k}+(t-|x-y_k|)\dot\zeta_{0k})}{4\pi|x-y_k|}
+\frac{(\zeta_{0k}+t\dot\zeta_{0k})e^{-m|x-y_k|}}{4\pi|x-y_k|}\\
&+&\frac{m}{4\pi}\int_0^t\frac{\theta(t-s-|x-y_k|)
J_1(m\sqrt{(t-s)^2-|x-y_k|^2})}{\sqrt{(t-s)^2-|x-y_k|^2}}(\zeta_{0k}+s\dot\zeta_{0k})ds,\quad t\ge 0.
\end{eqnarray*}
Therefore, for  $j\not=k$ there exist 
\begin{eqnarray}\nonumber
\lambda_{j,k}(t)&=&\lim_{x \to y_j}\,\psi_{f,k}(x,t)=\psi_{f,k}(y_j,t)=
-\frac{\theta(t-|y_j-y_k|)(\zeta_{0k}+(t-|y_j-y_k|)\dot\zeta_{0k})}{4\pi|y_j-y_k|}
+(\zeta_{0k}+t\dot\zeta_{0k})g_{jk}\\
\label{lamG1}
&+&\frac{m}{4\pi}\int_0^t\frac{\theta(t-s-|y_j-y_k|)
J_1(m\sqrt{(t-s)^2-|y_j-y_k|^2})}{\sqrt{(t-s)^2-|y_j-y_k|^2}}(\zeta_{0k}+s\dot\zeta_{0k})ds
\in C([0,\infty)).
\end{eqnarray}
Moreover, for any $t>0$ there exist (cf. (\ref{lamG}))
\begin{equation}\label{lamjj}
\lambda_{j,j}(t)=\lim_{x \to y_j}\,\psi_{f,j}(x,t)=
-\frac{m(\zeta_{0j}+t\dot\zeta_{0j})-\dot\zeta_0}{4\pi}
+\frac{m}{4\pi}\int_0^t \frac{J_1(m(t-s))}{(t-s)}(\zeta_{0j}+s\dot\zeta_{0j})ds.
\end{equation}
We set
\begin{equation}\label{lamjj0}
\lambda_{j,j}(0)=\lim_{t \to 0}\lambda_{j,j}(t)=-\frac{m\zeta_{0j}-\dot\zeta_{0j}}{4\pi}.
\end{equation}
Thus, (\ref{lamreg})--(\ref{lamjj0}) imply that
\begin{equation}\label{lamj}
\lambda_j(t):=\lim_{x \to y_j}\,\psi_{f}(x,t)=
\lambda_{j,reg}(t)+\sum\limits_{1\le k\le n}\lambda_{k,j}(t)\in C([0,\infty)).
\end{equation}
%%%%%%%%%%%%%%%%%%%%%%%%%%%%%%%%%%%%%%%%%%%%%%%%%%%%%%%%%%%%%
\begin{remark}\label{rm1}
{\it Evidently,
\begin{equation}\label{dotlamjj}
\dot\lambda_{j,j}(t)\in L^2_{loc}([0,\infty)).
\end{equation}
Nevertheless, $\dot\lambda_{k,j}(t)\not\in L^2_{loc}([0,\infty))$ for $k\not=j$, because $\dot\lambda_{k,j}(t)$ contains  the term}
\begin{equation}\label{mujj}
-\frac{\dot\theta(t-|y_j-y_k|)(\zeta_{0k}+(t-|y_j-y_k|)\dot\zeta_{0k})}{4\pi|y_j-y_k|}
=-\frac{\delta(t-|y_j-y_k|)\zeta_{0k}}{4\pi t}.
\end{equation}
\end{remark}
%%%%%%%%%%%%%%%%%%%%%%%%%%%%%%%%%%%%%%%%%%%%%%%%%%%%%%%%%%%%%%%%%%%%%%%%%%%%%%%%%%%%%%%%%%%%%%%%
\section{Nonlinear point interaction}
\label{nonlin-sect}
%%%%%%%%%%%%%%%%%%%%%%%%%%%%%%%%%%%%%%%%%%%%%%%%%%%%%%%%%%%%%%%%%%%%%%%%%%%%%%%%%%%%%%%%%%%%%%%%
%%%%%%%%%%%%%%%%%%%%%%%%%%%%%%%%%%%%%%%%%%%%%%%%%%%%%%%%%%%%%%%%%%%%%%%%%%%%%%%%%%%%%%%%%%%%%%%%%%%%%%%%%%%%%%%%%%%%
First we adjust the nonlinearity $F$ so that it becomes Lipschitz continuous.
Define
\begin{equation}\label{Lambda}
\Lambda(\Psi_0)=\sqrt{({\cal H}_F(\Psi_0)+a)/b},
\end{equation}
where $\Psi_0\in {\cal D}_F$ is the initial data from Theorem \ref{theorem-well-posedness} 
and $a$, $b$ are constants from (\ref{bound-below}).
Then we may pick a modified potential function
$\tilde U(\zeta)\in C^2(\C,\R)$, so that\\
i) the identity holds
\begin{equation}\label{Lambda1}
\tilde U(\zeta)= U(\zeta),\quad |\zeta|\le\Lambda(\Psi_0),
\end{equation}
ii) $\tilde U(\zeta)$ satisfies (\ref{bound-below}) with the same constant  $a$, $b$ as $U(\zeta)$ does:
\begin{equation}\label{Lambda2}
\tilde U(\zeta)-{\cal G}(\zeta)\ge b|\zeta|^2 -a,\quad \zeta\in\C,
\end{equation}
iii) the function $\tilde F=\pa_{\ov\zeta}\tilde U(\zeta)$ is Lipschitz continuous:
\begin{equation}\label{Lambda22}
|\tilde F(\zeta_1)-\tilde F(\zeta_2)|\le C|\zeta_1-\zeta_2|,\quad\zeta_1,\zeta_2\in\C.
\end{equation}
We suppose that $\Psi_0=(\psi_0,\pi_0)\in {\cal D}_{\tilde F}=D_{\tilde F}\oplus\dot D$,
and  consider the Cauchy problem for (\ref{KG})) with the modified nonlinearity $\tilde F$.
As before we denote by $\psi_f(x,t)\in C([0,\infty),L^2(\R^3)$  the unique solution to  (\ref{CP1}),
and $\lambda_j(t)$, $t\ge 0$ is defined by (\ref{lamj}).
The following lemma is proved by standard argument from the contraction mapping principle.
%%%%%%%%%%%%%%%%%%%%%%%%%%%%%%%%%%%%%%%%%%%%%%%%%%%%%%%%%%%%%%%%%%%%%%%%%%%%%%%%%%%%%%%%%%%%%%%%%%%%%%%%%%%%%
\begin{lemma}\label{LLWP}
Let conditions  (\ref {Lambda1})-(\ref {Lambda22}) be satisfies. 
Then there exists $\tau>0$ such that  the Cauchy problem
\begin{eqnarray}\nonumber
&&\frac {1}{4\pi}\dot\zeta_j(t)=\frac {m}{4\pi}\zeta_j(t)
+\sum\limits_{k\not =j}\frac{\theta(t-|y_j-y_k|)\zeta_k(t-|y_j-y_k|)}{4\pi|y_j-y_k|}+\lambda_j(t)\\
\label{delay}
&&-\sum\limits_{1\le k\le n}\frac{m}{4\pi}
\int_0^t\frac{\theta(t-s-|y_j-y_k|)J_1(m\sqrt{(t-s)^2-|y_j-y_k|^2})}{\sqrt{(t-s)^2-|y_j-y_k|^2}}\zeta_k(s)ds
-\tilde F_j(\zeta(t)),\quad \zeta_j(0)=\zeta_{0j}
\end{eqnarray}
has a unique solution $\zeta\in C([0,\tau])$.
\end{lemma}
Denote 
$$
\psi_j(t,x):=\frac{\theta(t-|x-y_j|)}{4\pi|x-y_j|}\zeta_j(t-|x-y_j|)-\frac{m}{4\pi}
\int_0^t\frac{\theta(t-s-|x-y_j|)
J_1(m\sqrt{(t-s)^2-|x-y_j|^2})}{\sqrt{(t-s)^2-|x-y_j|^2}}\zeta_j(s)ds, \quad t\in [0,\tau],
$$
with $\zeta_j$ from Lemma \ref{LLWP}.
Now we establish the local well-posedeness for (\ref{iKG}).
%%%%%%%%%%%%%%%%%%%%%%%%%%%%%%%%%%%%%%%%%%%%%%%%%%%%%%%%%%%%%%%%%%%%%%%%%%%%%%%%%%%%%%%%%%%%%%%%%%%%%%%%%%%%%%%%%%
\begin{proposition}\label{TLWP}(Local well-posedeness).
Let  the conditions  (\ref {Lambda1})--(\ref {Lambda22}) hold.
Then the function
\[
\psi(x,t):= \psi_f(x,t)+\sum\limits_{1\le j\le n}\psi_j(x,t)\in D_{\tilde F}, \quad t\in [0,\tau]
\]
is a unique strong solution to the Cauchy problem  
\begin{equation}\label{CP}
\ddot{\psi}(t) = (\Delta-m^2)\psi(t)+\sum\limits_{1\le j\le n}\zeta_j(t) g_j,\quad \psi(0) = \psi_0,\quad\dot\psi(0)  =  \pi_0,
\end{equation}
\begin{equation}\label{bc}
\lim_{x \to y_j}\, \left( \psi(x,t)-\zeta_j(t) g_j\right)=\tilde F_j(\zeta(t)),
\end{equation}
and
\[
\dot\psi(t)\in \dot D,~~\quad t\in [0,\tau].
\]
\end{proposition}
%%%%%%%%%%%%%%%%%%%%%%%%%%%%%%%%%%%%%%%%%%%%%%%%%%%%%%%%%%%%%%%%%%%%%%%%%%%%%%%%%%%%%%%%%%5
\begin{proof}
Since $\zeta_j(t)$ solves (\ref{delay}) one has similarly to (\ref{lim_z})
\begin{eqnarray}\nonumber
\lim_{x \to y_j}\, ( \psi(t,x)\!-\!\zeta_j(t) g_j(x))
\!\!\!\!\!\!\!\!\!\!&&\!\!\!=\lambda_j(t)+
\lim_{x\to y_j}\big(\sum\limits_{1\le k\le n}\frac{\theta(t-|x-y_k|)\zeta_k(t-|x-y_k|)}{4\pi|x-y_k|}
-\frac{\zeta_j(t)e^{-m|x-y_j|}}{4\pi|x-y_j|}\big)\\
\nonumber
&&\!\!\!-\lim_{x\to y_j}\sum\limits_{1\le k\le n}\frac{m}{4\pi}
\int_0^t\frac{\theta(t-s-|x-y_k|)J_1(m\sqrt{(t-s)^2-|x-y_k|^2})}{\sqrt{(t-s)^2-|x-y_k|^2}}\zeta_k(s)ds\Big)
\\
\nonumber
&&\!\!\!=\lambda_j(t)-\frac{1}{4\pi}(\dot\zeta_j(t)-m\zeta_j(t))
+\sum\limits_{k\not =j}\frac{\theta(t-|y_j-y_k|)\zeta_k(t-|y_j-y_k|)}{4\pi|y_j-y_k|}\\
\label{lim_zeta}
&&\!\!\!-\!\sum\limits_{1\le k\le n}\frac{m}{4\pi}
\int_0^t\frac{\theta(t-s-|y_j-y_k|)J_1(m\sqrt{(t-s)^2-|y_j-y_k|^2})}{\sqrt{(t-s)^2-|y_j-y_k|^2}}\zeta_k(s)ds
=\tilde F_j(\zeta(t))
\end{eqnarray}
and hence (\ref{bc}) are satisfied.  
Further,
$$
\ddot\psi=\ddot\psi_f+\sum\limits_{1\le j\le n}\ddot\psi_j=(\Delta-m^2)\psi_f
+\sum\limits_{1\le j\le n}(\Delta-m^2)\psi_j+\sum\limits_{1\le j\le n}\zeta_j\delta(\cdot-y_j)
=(\Delta-m^2)\psi+\sum\limits_{1\le j\le n}\zeta_j\delta(\cdot-y_j)
$$
and $\psi$ solves (\ref{CP}) then. 
Finally, let us prove that
\begin{equation}\label{psipsi}
(\psi_{reg}(t),\dot\psi_{reg}(t))\in {\cal X},\qquad~~~t\in [0,\tau].
\end{equation}
The function  $\psi_{reg}(x,t)=\psi(x,t)-\sum\limits_{1\le j\le n}\zeta_j(t) g_j(x)$ is a solution to 
\[
\ddot\psi_{reg}(x,t)=(\Delta-m^2)\psi_{reg}(x,t)-\sum\limits_{1\le j\le n}\ddot\zeta_j(t) g_j(x)
\]
with regular initial data $(\psi_{0,reg},\pi_{0,reg})\in {\cal X}$.
Due to (\ref{dotlamreg}), (\ref{lamj}), and (\ref{dotlamjj}) the  derivative with respect to $t$ of the RHS of
(\ref{delay}) belong to $L^2([0,\tau])$, since the terms with $\delta$-functions
cancel each other.
Hence,  $\ddot\zeta_j\in L^2([0,\tau])$, and  (\ref{psipsi}) holds by Lemma \ref{e-1}.

Suppose now that $\tilde\psi=\tilde\psi_{reg}+\sum\limits_{1\le j\le n}\tilde\zeta_j g_j$
is another strong solution of (\ref{CP}). 
Then, by reversing the above argument, the boundary
conditions (\ref{bc}) imply that $\tilde\zeta_j$ solves the Cauchy problem
(\ref{delay}). The uniqueness of the solution of (\ref{delay}) implies that
$\tilde\zeta_j=\zeta_j$. Then, defining
$$
\psi_j(t,x):=\frac{\theta(t-|x-y_j|)}{4\pi|x-y_j|}\zeta_j(t-|x-y_j|)-\frac{m}{4\pi}
\int_0^t\frac{\theta(t-s-|x-y_j|)J_1(m\sqrt{(t-s)^2-|x-y_j|^2})}
{\sqrt{(t-s)^2-|x-y_j|^2}}\zeta_j(s)ds, \quad t\in [0,\tau],
$$
for $\tilde\psi_f=\tilde\psi-\sum\limits_{1\le j\le n}\psi_j(t,x)$ one obtains
$$
\ddot{\tilde\psi}_f=(\Delta-m^2)\tilde\psi_{reg}-
\sum\limits_{1\le j\le n}((\Delta-m^2)\psi_j+\zeta_j\delta(\cdot-y_j))=
(\Delta-m^2)(\tilde\psi_{reg}-\sum\limits_{1\le j\le n}(\psi_j-\zeta_j g_j))=(\Delta-m^2)\tilde\psi_f\,,
$$
i.e $\tilde\psi_f$ solves the Cauchy problem (\ref{CP1}). 
Thus, by the uniqueness of the solution
(\ref{CP1}), $\tilde\psi_f=\psi_f$ and then $\tilde\psi=\psi$.
\end{proof}
%%%%%%%%%%%%%%%%%%%%%%%%%%%%%%%%%%%%%%%%%%%%%%%%%%%%%%%%%%%%%%%%%%%%%%%%%%%%%%%%%%%%%%%%%%%%%%%%%%%%%
\begin{lemma}\label{H-pr}
Let conditions (\ref{Lambda1})-(\ref{Lambda22}) hold, and let 
$\Psi(t)=(\psi(t),\dot\psi(t))\in {\cal D}_{\tilde F}$, $t\in [0,\tau]$, be a solution to
(\ref{CP})-(\ref{bc}). 
Then 
\begin{equation}\label{cHFT}
{\cal H}_{\tilde F}(\Psi(t))=\Vert\dot\psi(t)\Vert^2+\Vert\nabla\psi_{reg}(t)\Vert^2
+m^2\Vert\psi_{reg}(t)\Vert^2+\tilde U(\zeta(t))-{\cal G}(\zeta(t))=const,\quad t\in [0,\tau].
\end{equation}
\end{lemma}
%%%%%%%%%%%%%%%%%%%%%%%%%%%%%%%%%%%%%%%%%%%%%%%%%%%%%%%%%%%%%%%%%%%%%%%%%%%%%%%%%%%%%%%%%
\begin{proof}
Definition (\ref{HF}) of operator $H_{\tilde F}$ implies
\begin{eqnarray*}
\frac{d}{dt}\Vert\dot\psi\Vert^2
&=&\langle H_{\tilde F}\psi,\dot\psi\rangle+\langle\dot\psi,H_{\tilde F}\psi\rangle
=\langle(\Delta-m^2)\psi_{reg},\dot\psi_{reg}+\sum\limits_{1\le j\le n}\dot \zeta_j g_j\rangle
+\langle\dot\psi_{reg}+\sum\limits_{1\le j\le n}\dot \zeta_j g_j,(\Delta-m^2)\psi_{reg}\rangle\\
&=&\langle(\Delta-m^2)\psi_{reg},\dot\psi_{reg}\rangle+\langle\dot\psi_{reg},(\Delta-m^2)\psi_{reg}\rangle
-\sum\limits_{1\le j\le n}\dot{\overline\zeta}_{j}\psi_{reg}(y_j)
-\sum\limits_{1\le j\le n}\dot\zeta_j\overline\psi_{reg}(y_j)\\
&=&\frac{d}{dt}\Big(-\Vert\nabla\psi_{reg}\Vert^2-m^2\Vert\psi_{reg}\Vert^2\Big)
-\sum\limits_{1\le j\le n}\dot{\overline\zeta}_{j}{\tilde F}_j(\zeta)
-\sum\limits_{1\le j\le n}\dot\zeta _j\overline {\tilde F}_j(\zeta)
+\sum\limits_{j\not =k}g_{kj}(\dot{\overline\zeta}_{j}\zeta_k+\dot\zeta_j\overline\zeta_k)\\
&=&\frac{d}{dt}\Big(-\Vert\nabla\psi_{reg}\Vert^2-m^2\Vert\psi_{reg}\Vert^2-\tilde U(\zeta)
+{\cal G}(\zeta)\Big).
\end{eqnarray*}
Then (\ref{cHFT}) follows.
\end{proof}
%%%%%%%%%%%%%%%%%%%%%%%%%%%%%%%%%%%%%%%%%%%
\begin{corollary}\label{cor1}
The following identity holds
\begin{equation}\label{UtU}
\tilde U(\zeta(t))=U(\zeta(t)), \quad t\in [0,\tau].
\end{equation}
\end{corollary}
%%%%%%%%%%%%%%%%%%%%%%%%%%%%%%%%%%%%%%%%%%%%
\begin{proof}
First note that 
\[
{\cal H}_{F}(\Psi_0)\ge  U(\zeta_{0})-{\cal G}(\zeta_{0})\ge b|\zeta_{0}|^2-a.
\]
Therefore, $|\zeta_0|\le\Lambda(\Psi_0)$ and then $\tilde U(\zeta_0)=U(\zeta_0)$, 
${\cal H}_{\tilde F}(\Psi_0)={\cal H}_{F}(\Psi_0)$.
Further, 
\[
{\cal H}_{\tilde F}(\Psi(t))\ge \tilde U(\zeta(t))-{\cal G}(\zeta(t))
\ge b|\zeta(t)|^2-a,\quad t\in [0,\tau].
\]
Hence (\ref{cHFT}) implies that
\begin{equation}\label{zeta_bound}
|\zeta(t)|\le \sqrt{({\cal H}_{\tilde F}(\Psi(t))+a)/b}=\sqrt{({\cal H}_{\tilde F}(\Psi_0)+a)/b}
=\sqrt{({\cal H}_{F}(\Psi_0)+a)/b}=\Lambda(\Psi_0),\quad t\in [0,\tau].
\end{equation}
\end{proof}
From the identity  (\ref{UtU}) it follows that we can replace $\tilde F$ by $F$ 
in  Proposition \ref{TLWP} and in Lemma \ref{H-pr}.

{\bf Proof of Theorem \ref{theorem-well-posedness}}. 
The solution $\Psi(t)=(\psi(t),\dot\psi(t))\in {\cal D}_{F}$ constructed in Proposition  \ref{TLWP}
exists for $0\le t\le\tau$, where the time span $\tau$ in Lemma \ref{LLWP} depends only on $\Lambda(\Psi_0)$.
Hence, the bound (\ref{zeta_bound}) at $t=\tau$ allows us to extend the solution $\Psi$ to the time
interval $[\tau, 2\tau]$. We proceed by induction to obtain the solution for all $t\ge 0$.

\end{document}